\documentclass[11pt]{amsart}

\usepackage[english]{babel}
\usepackage{amsmath}
\usepackage{amsfonts}
\usepackage{amssymb}
\usepackage{amsthm}
\usepackage{epsfig}
\usepackage{graphicx,colortbl}
\usepackage{url,hyperref}
\usepackage{dsfont}
\usepackage{varwidth}

\newtheorem{theorem}{Theorem}[section]
\newtheorem{lemma}{Lemma}[section]
\newtheorem{corollary}{Corollary}[section]
\newtheorem{proposition}{Proposition}[section]

\newtheorem{conjecture}{Conjecture}
\newtheorem{question}{Question}
\newtheorem{remark}{Remark}[section]

\makeatletter
\newcommand*\rel@kern[1]{\kern#1\dimexpr\macc@kerna}
\newcommand*\widebar[1]{%
  \begingroup
  \def\mathaccent##1##2{%
    \rel@kern{0.8}%
    \overline{\rel@kern{-0.8}\macc@nucleus\rel@kern{0.2}}%
    \rel@kern{-0.2}%
  }%
  \macc@depth\@ne
  \let\math@bgroup\@empty \let\math@egroup\macc@set@skewchar
  \mathsurround\z@ \frozen@everymath{\mathgroup\macc@group\relax}%
  \macc@set@skewchar\relax
  \let\mathaccentV\macc@nested@a
  \macc@nested@a\relax111{#1}%
  \endgroup
}
\makeatother

\newcounter{theor}

\newtheorem{thm}[theor]{Theorem}

\newtheorem{cor}[theor]{Corollary}

\def\B{\mathbb{B}}

\def\R{\mathbb{R}}
\def\N{\mathbb{N}}

\def\vol{\mathrm{vol}}

\def\esc#1{\left\langle #1\right\rangle}

\newcommand{\dlat}{\mathrm{d} }

\def\esc#1{\left\langle #1\right\rangle}
\def\g{\mathrm{bar}}
\def\cov{\mathrm{Cov}}

\numberwithin{equation}{section}

\begin{document}

\title{Entropy, slicing problem and functional Mahler's conjecture}
\author[M. Fradelizi]{Matthieu Fradelizi}
\email{matthieu.fradelizi@univ-eiffel.fr}
\address{Univ Gustave Eiffel, Univ Paris Est Creteil, CNRS, LAMA UMR8050 F-77447 Marne-la-Vallée, France}

\author[F. Marín Sola]{Francisco Marín Sola}
\email{francisco.marin7@um.es}
\email{francisco.msola@cud.upct.es}
\address{Centro Universitario de la Defensa (CUD). Santiago de la Ribera, Murcia (España)}

\thanks{ The second named author is supported by the grant PID2021-124157NB-I00, funded by MCIN/AEI/10.13039/501100011033/``ERDF A way of making Europe'', as well as by the grant  ``Proyecto financiado por la CARM a través de la convocatoria de Ayudas a proyectos para el desarrollo de investigación científica y técnica por grupos competitivos, incluida en el Programa Regional de Fomento de la Investigación Científica y Técnica (Plan de Actuación 2022) de la Fundación Séneca-Agencia de Ciencia y Tecnología de la Región de Murcia, REF. 21899/PI/22''.}

\maketitle

\begin{abstract}
In a recent work, Bo'az Klartag showed that, given a convex body with minimal volume product, its isotropic constant is related to its volume product. As a consequence, he obtained that a strong version of the slicing conjecture implies Mahler's conjecture. In this work, we extend these geometrical results to the realm of log-concave functions. In this regard, the functional analogues of the projective perturbations of the body are the log-Laplace perturbations of the function. The differentiation along these transformations is simplified thanks to the known properties of the log-Laplace transform. Moreover, we show that achieving such an analogous result requires the consideration of the suitable version of the isotropic constant, notably the one incorporating the entropy. Finally, an investigation into the equivalences between the functional and geometrical strong forms of the slicing conjecture is provided.



\end{abstract}

\section{Introduction}
Let $K \subset \R^n$ be a convex body containing the origin in its interior. The \emph{polar body} of $K$ is defined by $K^{\circ} = \{y\in \R^n : \esc{x,y} \leq 1,\ \forall\  x \in K \}$. It is also a convex body containing the origin in its interior. Denote by $\vol(\cdot)$ the Lebesgue measure on $\R^n$. The \emph{Mahler volume} and the \emph{volume product} of $K$ are defined as
\begin{equation*}
   M(K) = \vol(K)\vol(K^\circ) \quad \hbox{and} \quad P(K) = \min_{z \in K} M(K-z).
\end{equation*}
It is known that the above minimum is attained at a unique point called the \emph{Santal\'o point} of $K$ denoted by $s(K)$. The Blaschke-Santaló inequality states that, for every convex body $K$,
$$
P(K) \leq P(B^n_2),
$$
where $B^n_2$ is the Euclidean unit ball in $\R^n$. The reverse inequality, known as \emph{Mahler's conjecture}, has been verified for different families of convex bodies (see e.g. \cite{FrMeZv} and also \cite[Section~10.7]{Sch2}), but in general it is still open. This conjecture claims that, for every convex body $K$,  
\begin{equation}\label{c:Mahler_Conj}
    P(K) \geq P(\Delta^n) = \frac{(n+1)^{n+1}}{(n!)^2},
\end{equation}
where $\Delta^n$ is a $n$-dimensional regular simplex with barycenter at the origin (see \cite{FrMeZv} for an overview of the rich theory developed to study the volume product). 
Since 
$$
\min_{K} P(K) = \min_{K}P(K-s(K)) = \min_{K}M(K),
$$
the conjecture equivalently postulates that $M(K) \geq M(\Delta_n)$, for every convex body $K$ containing the origin in its interior.
\smallskip

In \cite{Kl}, among other related topics, Klartag studied the relationship of Mahler's conjecture \eqref{c:Mahler_Conj} with the so-called slicing problem (also referred to as the slicing conjecture -see \cite{KlMi} for a survey on the topic). 
We recall the following definitions of the covariance matrix and the isotropic constant of a convex body $K$ in $\R^n$: 
\[
 \cov(K)=\int_K \frac{xx^T\dlat x}{\vol(K)}-\g(K)\g(K)^T\quad\hbox{and}\quad L_K=\frac{(\det\cov(K))^\frac{1}{2n}}{\vol(K)^\frac{1}{n}}.
\]
Recall also that the barycenter of a convex body $K$ is defined by 
\[
\g(K)=\int_K x\dlat x.
\]
The usual slicing conjecture postulates that there is a uniform upper bound of the isotropic constant of any convex body in any dimension, while its strong version asserts that, in any fixed dimension, the simplex maximizes the isotropic constant among convex bodies.
More precisely, using projective perturbations of an extremal body, Klartag proved the following (see also \cite{BaSoTz} for a shorter proof): 
\begin{thm}[Klartag \cite{Kl}]
    Let $K \subset \R^n$ be a convex body which is a local minimizer of the volume product. Then $ \text{\emph{Cov}}(K^{\circ}) \geq (n+2)^{-2}\, \text{\emph{Cov}}(K)^{-1}$ in the sense of symmetric matrices. 
\end{thm}
Taking determinants and using that, for $\Delta_n$, there is equality in the preceding inequality, it follows that, for $K$ being a local minimizer of the volume product, one has
\[
L_KL_{K^\circ} M(K)^\frac{1}{n}=\left(\det\cov(K)\det\cov(K^\circ)\right)^{1/2n}\ge\frac{1}{n+2}= L_{\Delta_n}L_{\Delta_n^\circ}M(\Delta_n)^\frac{1}{n}.
\]
Consequently, if, for any convex body $K$, $L_{K} \leq L_{\Delta^n}$, then Mahler's conjecture holds. Klartag thus deduced the following corollary.
\begin{cor}[Klartag \cite{Kl}]\label{c:slicing_implies_Mahler}
    The strong version of slicing's  conjecture for convex bodies in dimension $n$ implies Mahler's conjecture in dimension $n$. 
\end{cor}

\smallskip
In \cite{BaSoTz}, the above results of Klartag were reproved with a simpler proof. 
In this article, we establish the analogous statements for the functional forms of the conjectures: we study the relationship between the functional Mahler conjecture and the functional slicing conjecture for log-concave functions. Let $f:\R^n\to\R_{+}$ be an integrable log-concave function. The polar, the Mahler volume  and the volume product of $f$ are defined as 
\[
f^\circ(y)=\inf_x\frac{e^{-\langle x,y\rangle}}{f(x)},\quad M(f)=\int f\int f^\circ \quad\hbox{and}\quad P(f)=\inf_{z\in\R^n}M(f_z),
\]
where $f_z(x)=f(x-z)$.
The functional form of Mahler's conjecture for log-concave functions postulates that, among log-concave functions $P(f)\ge e^n$, with equality for 
\[
f_0(x)= e^{-\sum_{i=1}^{n} x_i}\chi_{_{[-1,+\infty)^n}}(x).
\]
Equivalently, as it happens in the case of convex bodies, it postulates that $M(f)\ge e^n$.
The barycenter and  covariance matrix of $f$ are 
\[
\g(f)=\frac{\int xf(x)dx}{\int f}
\quad\hbox{and}\quad 
\cov(f)=\frac{\int xx^Tf(x)dx}{\int f}-\g(f)\g(f)^T.
\]
The differential entropy and the varentropy of $f$ are  
\[
h(f)=-\frac{\int f\log(f)}{\int f} \quad\hbox{and}\quad 
V(f)=\frac{\int f\log(f)^2}{\int f} -(h(f))^2.
\]
The suitable (for us) isotropic constant of $f$ is
\[
\widehat{L}_f = \left(\frac{e^{-h(f)}}{\int f}\right)^{1/n}\bigl(\det\cov(f)\bigr)^{1/2n}.
\]
The usual slicing conjecture for log-concave function asks if their isotropic constants are upper bounded by an absolute universal constant, while its stronger form postulates that, in any fixed dimension $n$, for any log-concave function $f$ on $\R^n$, one has $\widehat{L}_f\le \widehat{L}_{f_0}=1/e$. 
Our main result is the following theorem.
\begin{theorem}\label{t:Cov_Matrix_Ineq}
    Let $f:\R^n \to \R_+$ be an integrable log-concave function. If $f$ is a local minimizer of the functional Mahler volume $M(f)$, then 
\begin{enumerate}
    \item $f$ is a critical point so that  
\[
\g(f)=0,\quad \g(f^\circ)=0\quad\hbox{and}\quad  h(f)+h(f^\circ)=n.
\]
    \item $f$ is a local minimizer so that, in the sense of symmetric matrices, 
\[
\emph{Cov}(f^{\circ}) \geq \emph{Cov}(f)^{-1} \quad\hbox{and}\quad V(f)+V(f^\circ)\ge n.
\]
\end{enumerate}
\end{theorem}
As for the case of convex bodies, if $f$ is a local minimizer of the functional Mahler volume, taking determinants of the previous inequality, we get that 
\[
\widehat{L}_f\widehat{L}_{f^\circ}e^\frac{h(f)+h(f^\circ)}{n}M(f)^\frac{1}{n}=\bigl(\det\cov(f)\det\cov(f^\circ)\bigr)^{1/2n}\ge1.
\]
Using that $h(f)+h(f^\circ)=n$, we deduce that $\widehat{L}_f\widehat{L}_{f^\circ}M(f)^\frac{1}{n}\ge 1/e$. 
Hence, if $\widehat{L}_f\le \widehat{L}_{f_0}=1/e$, for any log-concave function $f$, which is the strong form of the slicing conjecture for log-concave functions mentioned above, then we conclude that $M(f)\ge M(f_0)=e^n$. Thus we get the following corollary.
\begin{corollary}\label{c:funct_slicing_implies_Mahler}
  The strong version of the slicing conjecture for log-concave functions in dimension $n$ implies Mahler's conjecture for log-concave functions in dimension $n$. 
\end{corollary}

This corollary, as well as its analogue for convex bodies, due to Klartag, attracts the attention to the strong version of the slicing conjectures. As for Mahler's conjecture, there are also a symmetric versions of the conjectures which are the following: is it true that, for any symmetric convex body $K$ in dimension $n$, one has $L_K\le L_{[-1,1]^n}=1/\sqrt{12}$? The functional log-concave analogue asks if, among log-concave even functions, one has $\widehat{L}_f\le \widehat{L}_{f_{\infty}}=1/\sqrt{12}$, where $f_{\infty}(x)=\chi_{_{[-1,1]^n}}(x)$.
\\

Notice that the strong version of the slicing conjecture for general convex bodies was considered by Campi, Colesanti and Gronchi \cite{CCG} in the form of the so-called Sylvester's problem, where they also proved the case $n=2$. Then,  Bisztriczky and  B\"{o}r\"{o}czky \cite{bis-bor}  established the symmetric case of the conjecture in dimension two and Meckes \cite{Meck} generalized this result to more general symmetric Sylvester's type problems. More recently, the functional analogue of the strong slicing conjecture for log-concave functions was  proved in dimension one for even functions by  Madiman, Nayar and Tkocz in \cite{mnt} and in 2023 for the general case, by  Melbourne,  Nayar and  Roberto in \cite{MeNaRo}. 
\\

As a second objective of this paper, we shall delve into the equivalences between the geometrical and functional strong forms of the slicing conjecture. For instance, among other results, we show that the strong slicing conjecture for log-concave functions in $\R^{n+1}$ implies the strong slicing conjecture for convex bodies in $\R^n$ (see Proposition \ref{p:functional_implies_geometrical}).
\\

The paper is organized as follows: in Section \ref{s:prelim} we recall some preliminaries and we fix some notation that will be used later on. Section \ref{s:discussion_slicing} is devoted to a brief discussion on the different versions of the functional strong slicing conjecture, and to show which one provides, via Theorem \ref{t:Cov_Matrix_Ineq}, an analogue of Klartag's statement in this setting. The proofs of our main results are collected in Section \ref{s:main} and along Section \ref{s:equivalences_between_conjectures}, we present an investigation into the correlations between functional and geometrical conjectures. 

\section{Preliminaries}\label{s:prelim}

We shall work in the $n$-dimensional Euclidean space $\R^n$  with the standard inner product $\esc{\cdot,\cdot}$, and  $x_{i}$ are used for the $i$-th coordinate of a vector $x \in \R^n$.  Given any set $M \subset \R^n$, we  use $\chi_{_{M}}$ to denote its characteristic function. Furthermore, $\|x\|_{K} = \inf \{\lambda > 0 :   x \in \lambda K\}$ is the gauge function of a convex body $K$ containing the origin in its interior.
\\

A function $f : \R^n \to \R_{+}$ is said to be \emph{log-concave} if $f(x) = e^{-\varphi(x)}$, where $\varphi : \R^n \to \R \cup \{+\infty\}$ is a convex function. Let us recall that, for any integrable log-concave function $f$, the \emph{functional Mahler volume} and the \emph{functional volume product} are defined as
\begin{equation}
    M(f) = \int_{\R^n} e^{-\varphi(x)}\,\dlat x \int_{\R^n} e^{-\mathcal{L}\varphi(x)}\,\dlat x \quad  \text{and} \quad  P(f)=\inf_{z\in\R^n}M(f_z),
\end{equation}
where $f_z(x)=f(x-z)$, and $\mathcal{L}\varphi(x) : \R^n \to \R_{+}$ is the Legendre transform of $\varphi$ which is given by
$$
\mathcal{L}\varphi(y) = \sup_{x} \esc{x,y} - \varphi(x).
$$
Note that we write $f^{\circ}(x) = e^{-\mathcal{L}\varphi(x)}$. The functional Blaschke-Santaló inequality for log-concave functions, proven in \cite{Ball} for even functions (see also \cite{NakTsu} for a new proof using semi-groups), and obtained in full generality in \cite{AaKlMi}, states that
$$
P(f) \leq P(e^{-\frac{|x|^2}{2}}) = (2\pi)^n.
$$
As mentioned in the introduction, its reverse counterpart, known as Mahler's conjecture, claims that (see e.g \cite{FrMe2}) if $f : \R^n \to \R_{+}$ is an integrable log-concave  function, then 
\begin{equation}\label{e:func_Mahler}
P(f) \geq P\left(e^{-\sum_{i=1}^{n} x_i}\chi_{_{[-1,+\infty)^n}}(x)\right) = e^n.
\end{equation}
When $f$ is unconditional, it was proved in \cite{FrMe1} that
\begin{equation}\label{even:func_Mahler}
P(f) \geq P\left(e^{-\sum_{i=1}^{n} |x_i|}\right) = 4^n.
\end{equation}
The case $n=1$ in \eqref{e:func_Mahler} was also verified in \cite{FrMe2}. For even functions, the case $n=2$ of inequality \eqref{even:func_Mahler} was proved in \cite{FN}.
\smallskip



For the reader's convenience, we finish this section by recalling the two strong forms of the slicing conjecture (both the symmetric and the non-symmetric case) in the geometrical setting. 
\begin{conjecture}\label{GC:symmetric}
    Let $K \subset \R^n$ be a centrally symmetric convex body (i.e., $K = -K$). Then
    \begin{equation*}
        L_{K} \leq L_{[-1,1]^n} = \frac{1}{\sqrt{12}}.
    \end{equation*}
\end{conjecture}

\begin{conjecture}\label{GC:non_symmetric}
    Let $K \subset \R^n$ be a convex body. Then
    \begin{equation*}
         L_K \leq L_{\Delta^n} = \frac{(n!)^{1/n}}{(n+1)^{(n+1)/2n}\sqrt{n + 2}}.
    \end{equation*}
\end{conjecture}

\section{On the functional isotropic constant}\label{s:discussion_slicing}


As mentioned earlier, in order to prove Theorem \ref{t:Cov_Matrix_Ineq} one has to use a specific definition of the isotropic constant. To explain this, let us first recall that, given an integrable centered $\log$-concave function with non-zero integral $f:\R^n \to \R_{+}$, the following definitions of the isotropic constant are contemplated in the literature:
\begin{equation*}
    \begin{split}
        L_{f} &= \left(\frac{\max_{x} f(x)}{\int_{\R^n}f(x) \,\dlat x}\right)^{1/n}\bigl(\det\cov(f)\bigr)^{1/2n},\\
        \widetilde{L}_{f} &= \left(\frac{f(0)}{\int_{\R^n}f(x) \,\dlat x}\right)^{1/n}\bigl(\det\cov(f)\bigr)^{1/2n} \quad \text{and}\\
        \widehat{L}_f &= \left(\frac{e^{-h(f)}}{\int_{\R^n}f(x) \,\dlat x}\right)^{1/n}\bigl(\det\cov(f)\bigr)^{1/2n}.
    \end{split}
\end{equation*}


We would also like to point out that the precise extremizers of the isotropic constant depend on the definition. On the one hand, the minimizer of $\widehat{L}_f$ is well known as it boils down to the study of entropy maximizers among log-concave functions with fixed covariance matrix. In fact, for every integrable centered log-concave function $f:\R^n \to \R_+$, one has as a consequence of Jensen's inequality that
$$
h(f) \leq h(\gamma) = \frac{n}{2}\log\left(2\pi e \sigma_2(f)^2\right),
$$
where $\gamma(x) = \frac{1}{(2\pi \sigma_2(f)^2)^{n/2}}e^{-\frac{|x|^2}{2 \sigma_2(f)^2}}$, and $\sigma_2(f) = \det \cov(f)^{1/2n}$. Thus one gets that
$$
\widehat{L}_f \geq \widehat{L}_{\gamma} = \frac{1}{\sqrt{2\pi e}}.
$$

On the other hand, in contrast with the previous situation, it was proved in \cite{Hen} (see also \cite[Proposition~2.3.12]{BrGiVaVr}) that
$$
L_f \geq L_{\chi_{_{B_2^n}}} = \frac{1}{\sqrt{n+2}}\vol{(B_2^n)}^{-1/n}.
$$
As far as we are aware, the minimizer of $\widetilde{L}_f$ is only known in dimension one. In this regard, it was shown in \cite{Fr2} that, for $f:\R\to\R_+$ log-concave, one has
$$
\widetilde{L}_{f} \geq \widetilde{L}_{\chi_{_{[-1,1]}}} = \frac{1}{\sqrt{12}}.
$$

The much more delicate study of maximizers leads us to the so-called strong slicing conjecture. We shall see that, as for the minimizers and based on the analysis of the one dimensional known results, the conjectured maximizers depend on the definition of the functional isotropic constant. Therefore, we shall contemplate different versions of this conjecture. 

 \begin{conjecture}\label{c:Functional_Slicing} Let $f:\R^{n} \to \R_{+}$ be an  integrable log-concave  function,  and $f_0,f_1, f_{\infty} : \R^n \to \R_{+}$ be the functions given by 
 $$ 
 f_0(x)= e^{-\sum_{i=1}^{n} x_i}\chi_{_{[-1,+\infty)^n}}(x), \quad f_1 (x) = e^{-\sum_{i=1}^{n} |x_i|} \quad \text{and} \quad f_{\infty}(x) = \chi_{_{[-1,1]^n}}(x).
 $$ Then
 \vspace{0.05cm}
    \begin{center}
        \begin{minipage}{2in}
             \begin{enumerate}
         \item $L_{f} \leq L_{f_0} = 1,$ 
\bigskip
         \item $\widetilde{L}_{f} \leq \widetilde{L}_{f_1} = \frac{1}{\sqrt{2}},$
\bigskip        
         \item $\widehat{L}_f \leq \widehat{L}_{f_0} = e^{-1}.$     
     \end{enumerate}
        \end{minipage}
    \end{center}
 \end{conjecture}

\begin{remark}
Note that, if $f$ is centered, using Jensen's inequality and \cite[Theorem~1]{Fr} it follows that
\begin{equation}\label{e:Rela_isotrop}
\frac{\max_x f(x)}{e^n} \leq e^{-h(f)} \leq f(0).
\end{equation}
As a consequence, one immediately obtains that
$$
e^{-1}L_{f} \leq \widehat{L}_f \leq \widetilde{L}_{f} \leq L_{f}.
$$
Thus,  Conjecture \ref{c:Functional_Slicing} iii) implies Conjecture \ref{c:Functional_Slicing} i).
\end{remark}

 These conjectures are supported by the fact that they hold in dimension one. As a matter of fact, on the one hand, from \cite[Theorem~8]{Fr2}, for any integrable log-concave  function $f : \R\to \R_{+}$,
$$
\sigma_2(f) \leq \frac{\int_{\R} f(t) \,\dlat t}{\max_{t} f(t)},
$$
with equality for the function $f_0(t) = e^{-t}\chi_{_{[-1,\infty)}}(t)$. On the other hand, using again \cite[Theorem~8]{Fr2}, one has
$$
\sigma_2(f) \leq \frac{\int_{\R} f(t) \,\dlat t}{2f(0)},
$$
with equality for the function $f_1(t) = e^{-|t|}$. Therefore, we get i) and ii) in dimension 1:
$$
L_{f} \leq L_{f_0} = 1 \quad \text{and} \quad \widetilde{L}_{f} \leq \widetilde{L}_{f_1} = \frac{1}{\sqrt{2}}.
$$
Moreover, very recently, Melbourne, Nayar and Roberto proved in \cite{MeNaRo} that, for any log-concave density function $f : \R\to \R_{+}$, one has 
$$
h(f) \geq \log \sigma_2(f) + 1,
$$
with equality for $f_0(t)= e^{-t}\chi_{_{[-1,\infty)}}(t)$. Hence, we get iii) in dimension 1:
$$
\widehat{L}_f \leq \widehat{L}_{f_0} = e^{-1}.
$$
 



We also consider the corresponding conjecture for even functions. Notice that, in this case $\max f=f(0)$ so that $L_f=\widetilde{L}_{f}$.
\begin{conjecture}\label{c:Functional_Slicing_even}
Let $f:\R^{n} \to \R_{+}$ be an integrable log-concave even function
\vspace{0.05cm}
    \begin{center}
        \begin{minipage}{2in}
             \begin{enumerate}
         \item $L_{f} \leq L_{f_1} = \frac{1}{\sqrt{2}},$ 
\bigskip        
         \item $\widehat{L}_f \leq \widehat{L}_{f_{\infty}} = \frac{1}{\sqrt{12}}.$    
     \end{enumerate}
        \end{minipage}
    \end{center}
\end{conjecture}

Again, these conjectures are known for $n=1$: the case i) was proved in \cite{Hen}, while
Conjecture \ref{c:Functional_Slicing_even} ii) was proved in dimension one by Madiman, Nayar and Tkocz in \cite{mnt} getting that
$$
h(f) \geq \log \sigma_2(f) + \log \sqrt{12},
$$
and thus,
$$
\widehat{L}_f \leq \frac{1}{\sqrt{12}},
$$
with equality for $f_{\infty}(t) = \chi_{_{[-1,1]}}(t)$.
\smallskip

As mentioned in the introduction, once the precise statements of the functional analogues of the strong slicing conjecture are established, our aim is to show why Theorem \ref{t:Cov_Matrix_Ineq} works only for the definition of the isotropic constant involving the entropy.
\smallskip

To this purpose, let $f:\R^{n} \to \R_{+}$ be an integrable log-concave function that is a local minimizer of the functional Mahler volume. On the one hand, if Conjecture \ref{c:Functional_Slicing}  i) holds true, the best we can get is that
\begin{equation*}
\begin{split}
   M(f) &\geq \frac{\max_{x}f(x)\max_{x}f^{\circ}(x)}{(L_{f_0}L_{f^{\circ}_0})^{n}} = \max_{x}f(x)\max_{x}f^{\circ}(x) \geq 1.
   \end{split}
\end{equation*}
On the other hand, assuming that Conjecture \ref{c:Functional_Slicing}  ii) stands, one can only ensure the following:
\begin{equation*}
    M(f) \geq \frac{f(0)f^{\circ}(0)}{(\widetilde{L}_{f_1}\widetilde{L}_{f^{\circ}_1})^{n}} = 2^{n/2}f(0)f^{\circ}(0) \geq 2^{n/2} e^{-n}.
\end{equation*}

Nonetheless, if Conjecture \ref{c:Functional_Slicing} iii) is true, we can deduce that
\begin{equation*}
    M(f) \geq \frac{e^{-\bigl(h(f) + h(f^{\circ})\bigr)}}{(\widehat{L}_{f_0}\widehat{L}_{f^{\circ}_0})^{n}} = e^{2n -\bigl(h(f) + h(f^{\circ})\bigr)}.
\end{equation*}
Hence, as mentioned in the introduction, Theorem \ref{t:Cov_Matrix_Ineq} shows that $h(f)+ h(f^{\circ}) = n$ and thus Mahler's conjecture follows. Therefore, on this matter, it seems natural to consider $\widehat{L}_f$.
\smallskip

Finally, as a matter of fact, we have from \eqref{e:Rela_isotrop} that $h(f) + h(f^{\circ}) \leq 2n$. Thus, since Theorem \ref{t:Cov_Matrix_Ineq} ensures that, for every local minimizer of the functional volume product, $h(f) + h(f^{\circ}) =n$,  the following question may be natural in this context.

\begin{question}\label{c:Conjecture4}
Let $f: \R^n\to \R_{+}$ be an integrable log-concave function. Does it hold that
   \begin{equation*}
       h(f) + h(f^{\circ}) \leq n?
   \end{equation*}
\end{question}

Unfortunately, this inequality doesn't hold, as shown by the following example.
Let $f(x)=\chi_{[-1,1]}(x)+e^{1-|x|}\chi_{(1,+\infty)}(|x|)$. Then, it is not difficult to see that $f^\circ(y)=e^{-|y|}\chi_{[-1,1]}(y)$. We shall see that, for $t$ large enough, one has $h(f^t)+h((f^t)^\circ)>1$, and hence the answer to Question \ref{c:Conjecture4} is negative. A quick calculation shows that $\int f^t=2(1+\frac{1}{t})$ and $h(f^t)=\frac{1}{t+1}$. One also has $h((f^t)^\circ)=1-\frac{t}{e^t-1}$. Therefore, the inequality $h(f^t)+h((f^t)^\circ)\le1$ would be equivalent to 
\[
\frac{1}{t+1}+1-\frac{t}{e^t-1}\le1.
\]
This in turn is equivalent to $e^t\le 1+t+t^2$ which doesn't hold for $t$ large enough.

\section{Proof of the main results}\label{s:main}
In order to prove Theorem \ref{t:Cov_Matrix_Ineq}, we follow an adaptation of the strategy used in \cite{BaSoTz}. To this purpose, let $f : \R^n \to \R_+$ be a log-concave integrable function. For any $x,y\in\R^n$, we define the following log-concave perturbation of $f$, $f_{x,y} : \R^n \to \R_{+}$ defined by
\[
f_{x,y}(z)=f(z-x)e^{-\langle z,y\rangle}\quad\hbox{thus}\quad 
f_{x,y}^\circ(w)=f^\circ(w-y)e^{-\esc{x,w-y}}.
\]
Let $F:\R^n \times \R^n \to \R_{+}$ be the function defined as $F(x,y)=M(f_{x,y})$. After changing variables, we get
$$
F(x,y)=e^{-\esc{x,y}}\int_{\R^n} e^{-\esc{y,z}} f(z)\,\dlat z \int_{\R^n} e^{-\esc{z,x}}f^\circ(z) \,\dlat z.
$$
Recall that the log-Laplace transform of a non-negative function $f$, denoted by $\Lambda_f: \R^n\to\R\cup\{+\infty\}$, is defined, for $x\in\R^n$, by
$$
\Lambda_f (x) = \log \int_{\R^n} e^{-\esc{x,z}}f(z)\,\dlat z.
$$
Therefore, since we want to study the minimum of the function $F$, we shall work with the function $\widetilde{F}:\R^n \times \R^n \to \R_{+}$ given by
\[
 \widetilde{F}(x,y) = \log F(x,y)
        = -\esc{x,y} + \Lambda_f(y)+ \Lambda_{f^\circ}(x). 
\]
This function is thus expressed in terms of the logarithmic Laplace transform of $f$ and $f^\circ$, which has been widely use to study log-concave functions and convex bodies (see \cite{KlMi2} and \cite[Chapter~7]{BrGiVaVr}). We collect in the following lemma two of its properties that we use (see e.g. \cite[Proposition~7.2.1]{BrGiVaVr}).
\begin{lemma}\label{l:properties_laplace_transform}
Let $f : \R^n \to \R_+$ be a log-concave integrable function. Let $\Lambda_f: \R^n \to \R$ be the function given by
$$
\Lambda_f (x) = \log \int_{\R^n} e^{-\esc{x,z}}f(z)\,\dlat z,
$$
and let $\mu_{f,x}$ be the probability measure with density function
$$
d\mu_{f,x}(z) = \frac{e^{-\esc{x,z}}f(z)}{e^{\Lambda_f (x)}}.
$$
Then, for any $x \in \R^n$,
\begin{equation*}
    \nabla\Lambda_f (x) = -\g(\mu_{f,x})\quad\hbox{and}\quad
    \text{\emph{Hess}}\, \Lambda_f (x) = \text{\emph{Cov}}(\mu_{f,x}).
\end{equation*}
\end{lemma}
We can now proceed with the proof of Theorem \ref{t:Cov_Matrix_Ineq}.
\begin{proof}[Proof of Theorem \ref{t:Cov_Matrix_Ineq}]
Let $f: \R^n \to \R_+$ be an integrable log-concave function which is a local minimizer of the functional volume product. From Lemma \ref{l:properties_laplace_transform} we get, for every $x,y \in \R^n$, that
      \begin{equation*}
    \begin{split}
    &\nabla_{x}\widetilde{F}(x,y) = -\bigl(y + \g(\mu_{f^\circ,x})\bigr) \quad \text{and}\\
    &\nabla_{y}\widetilde{F}(x,y) = -\bigl(x+ \g(\mu_{f,y})\bigr).
     \end{split}
    \end{equation*}
The latter, together with the second statement of Lemma \ref{l:properties_laplace_transform}, yields that
\begin{equation*}
        \mathrm{Hess}_{(x,y)} \widetilde{F}= \begin{pmatrix}\begin{array}{c|c}
\cov(\mu_{f^\circ,x}) & -I \\
\hline
-I & \cov(\mu_{f,y})
\end{array}
\end{pmatrix}.
    \end{equation*}
Now, taking into account that $\widetilde{F} = \log F$, we deduce that
\begin{equation*}
        \begin{split}
            &\mathrm{Jac}_{(0,0)}F = -M(f)\bigl(\g(f^{\circ}) \; , \;  \g(f) \bigr)\\
        \end{split}
    \end{equation*}
and
    \begin{equation*}
        \mathrm{Hess}_{(0,0)} F= M(f)\begin{pmatrix}\begin{array}{c|c}
\mathrm{Cov}(f^{\circ}) & -I \\
\hline
-I & \mathrm{Cov}(f)
\end{array}
\end{pmatrix}.
    \end{equation*}
Finally, since $\mathrm{Cov}(f^{\circ})$ and $\mathrm{Cov}(f)$ are two symmetric positive definite matrices, $\mathrm{Hess}_{(0,0)}F \geq 0$ if and only if $\mathrm{Cov}(f^{\circ}) \geq \mathrm{Cov}(f)^{-1}$.
\smallskip

To end the proof, we shall study another perturbation of the function $f$. In this regard, we define the function $p:\R_{+}\to \R_{+}$ given by 
\[
p(t)=M\left(f^t\right)=\int_{\R^n} f^t \int_{\R^n} \left(f^t\right)^\circ.
\]
Using that, for every $t>0$, 
\[
\left(f^t\right)^\circ(y)=\inf_x\frac{e^{-\esc{x,y}}}{f^t(x)}=\left(\inf_x\frac{e^{-\esc{x,\frac{y}{t}}}}{f(x)}\right)^t=\left(f^\circ\left(\frac{y}{t}\right)\right)^t
\]
and changing variables, one deduces that 
\[
p(t)=t^n\int_{\R^n} f^t\,  \int_{\R^n} (f^\circ)^t\, .
\]
Since $f$ is a local minimizer of the functional Mahler volume, then $p(t)\ge p(1)$, for all $t>0$ thus $(\log p)'(1)=0$ and $(\log p)''(1)\ge0$. Moreover, since
\[
(\log p)'(t)=\frac{n}{t}+\frac{\int_{\R^n} (\log f) f^t}{\int_{\R^n} f^t}+\frac{\int_{\R^n}(\log f^\circ) (f^\circ)^t}{\int_{\R^n} (f^\circ)^t},
\]
and
\[
\begin{split}
(\log p)''(t)&=-\frac{n}{t^2}+\frac{\int_{\R^n} (\log f)^2 f^t}{\int_{\R^n} f^t}+\left(\frac{\int_{\R^n} (\log f) f^t}{\int_{\R^n} f^t}\right)^2\\
&+ \frac{\int_{\R^n} (\log f^\circ)^2 (f^\circ)^t}{\int_{\R^n} (f^\circ)^t}+\left(\frac{\int_{\R^n} (\log f^\circ) (f^\circ)^t}{\int_{\R^n} (f^\circ)^t}\right)^2.
\end{split}
\]
The fact that $(\log p)'(1)=0$ and $(\log p)''(1)\ge0$ gives 
\[
h(f)+h(f^\circ)=n\quad\hbox{and}\quad V(f)+V(f^\circ)\ge n.
\]
\end{proof}

\section{Relationship between geometrical and functional conjectures}\label{s:equivalences_between_conjectures}


This section, in the spirit of what was done for Mahler's conjectures in \cite{FrMe2}, is devoted to study the correlation between the functional and geometrical forms of the strong slicing conjecture. We start by studying the relationships between the conjectures for even functions and centrally symmetric convex bodies.
\begin{proposition}\label{p:geometrical_implies_functional}
\noindent
\begin{enumerate}


     \item Let $n\geq 1$. If for any $m\in\N$, the strong slicing conjecture  (Conjecture \ref{GC:symmetric}) holds true for every centrally symmetric convex body in $\R^{n+m}$, then the functional strong slicing conjecture  holds also for every integrable log-concave even function in $\R^n$ (Conjecture \ref{c:Functional_Slicing_even} ii)).
     \smallskip
     
     \item Let $n\geq 1$. If the functional strong slicing conjecture holds true for every integrable log-concave even function in $\R^n$ (Conjecture \ref{c:Functional_Slicing_even} iii)), then the strong slicing conjecture  holds also for every centrally symmetric convex body in $\R^{n}$ (Conjecture \ref{GC:symmetric}).
     
\end{enumerate}
\end{proposition}

The general cases are collected in the next proposition. 
\begin{proposition}\label{p:functional_implies_geometrical}
\noindent
    \begin{enumerate}
    

    \item Let $n\geq 1$. If for any $m \in \mathbb{N}$, the strong slicing conjecture (Conjecture \ref{GC:non_symmetric}) holds true for every convex body in $\R^{n+m}$, then the strong functional slicing conjecture holds also for every integrable log-concave function in $\R^n$ (Conjecture \ref{c:Functional_Slicing} iii)).
    \smallskip

    \item Let $n\geq 1$. If the strong functional slicing conjecture (Conjecture \ref{c:Functional_Slicing} iii)) holds true for every integrable log-concave function in $\R^{n+1}$, then the strong slicing conjecture holds also for 
    every  convex body in $\R^n$ (Conjecture \ref{GC:non_symmetric}).
\end{enumerate}
\end{proposition}

\begin{remark}
    Note that Proposition \ref{p:functional_implies_geometrical} together with Theorem \ref{t:Cov_Matrix_Ineq} implies a weaker version of Corollary \ref{c:slicing_implies_Mahler}. Specifically, one can deduce that, if for every $n\geq 1$ and $m \in \N$ the strong slicing conjecture holds true for every convex body in $\R^{n+m}$, then Mahler's conjecture holds true for every convex body $K \in \R^n$.
\end{remark}



    

The main idea to prove Propositions \ref{p:geometrical_implies_functional} and \ref{p:functional_implies_geometrical} is to use a certain type of convex bodies which have already appeared in the literature (see e.g. \cite{AaKlMi} and \cite{FrMe2}). We recall their definition, together with some properties that we shall use, in the following lemma.
\begin{lemma}\label{l:iso_constant_of_revbodies}
    Let $C \subset \R^m$ be a centered convex body, and let $g:\R^n \to \R_{+}$ be an integrable function, which is concave on its support. Let $K_{m}(C,g) \subset \R^m \times \R^n$ be the convex body given by
    $$
    K_{m}(C,g) = \{ (y,x) \in \R^m \times \R^n: \|y\|_{C} \leq g(x) \}.
    $$
    Then we have 
    $$
    L_{K_{m}(C,g)}^{2(m+n)} = \frac{\left(\int_{\R^n} g(x)^{m+2} \,\dlat x\right)^m}{\left(\int_{\R^n} g(x)^m\,\dlat x\right)^{m+2}} \det\text{\emph{Cov}}(g^m) L_{C}^{2m}.
    $$
Moreover, if $f: \R^n \to \R_+$ is an integrable log-concave function, let $f_m : \R^n \to \R_+$ be the concave function given by $f_m (x) = \bigl(1 + \frac{1}{m}\log f(x)\bigr)_+$, and let $C_m \subset \R^m$ be a sequence of convex bodies. Then
    \begin{equation}\label{e:limit}
        \lim_{m \to \infty} \frac{L_{K_{m}(C_m,f_m)}^{2(m+n)}}{L_{C_m}^{2m}} = \widehat{L}_f^{2n}.
    \end{equation}
\end{lemma}
\begin{remark}
         Note  that \eqref{e:limit}  provides another justification for involving the entropy on the definition of the functional isotropic constant.
    \end{remark}
\begin{proof}
    First, using Fubini's theorem we have 
    $$
    \vol(K_{m}(C,g)) = \int_{\R^n}\int_{g(x)C}\dlat y \,\dlat x = \vol_{m}(C) \int_{\R^n} g(x)^m \,\dlat x. 
    $$
    Following the same idea, it is easy to see that for every $i,j$ we have 
    \begin{equation*}
        \begin{split}
            &\int_{K_{m}(C,g)} y_i \,y_j \,\dlat x \, \dlat y = \int_{C} y_i\,y_j \,\dlat y \int_{\R^n} g(x)^{m+2}\,\dlat x, \\
            &\int_{K_{m}(C,g)} x_i \,x_j \,\dlat x \, \dlat y = \vol_{m}(C) \int_{\R^n} x_i \, x_j g(x)^{m}\,\dlat x,\\
            &\int_{K_{m}(C,g)} y_i \,x_j \,\dlat x \, \dlat y = \int_{C} y_i\,\dlat y \int_{\R^n} x_j g(x)^{m+1}\,\dlat x.
        \end{split}
    \end{equation*}
    Hence, we get $\g\bigl(K_{m}(C,g)\bigr) = \bigl(0,\g(g^m)\bigr)$. Moreover, 
    \begin{equation*}
    \cov\bigl(K_{m}(C,g)\bigr) = \begin{pmatrix}
    \begin{array}{c|c}
  \frac{\int_{\R^n} g(x)^{m+2} \,\dlat x}{\int_{\R^n} g(x)^m\,\dlat x} \cov(C) & 0 \hspace{0.5cm}\\
 \hline
  0 & \cov(g^m)\\
 \end{array}
\end{pmatrix},
\end{equation*}
which implies the first statement. 

Finally, straightforward computations show that 
$$
\lim_{m \to \infty}\frac{\left(\int_{\R^n} f_m(x)^{(m+2)/m} \,\dlat x\right)^m}{\left(\int_{\R^n} f_m(x)\,\dlat x\right)^{m+2}} = \frac{e^{-2h(f)}}{\left(\int f(x)\,\dlat x\right)^2}.
$$
Thus \eqref{e:limit} follows.
\end{proof}

We now use a log-concave function which has already appeared in the literature (see \cite[Proposition~2]{FrMe2}).
\begin{lemma}\label{l:relationship_conjectures}
    Let $K \subset \R^n$ be a centered convex body, and let $f : \R^n \times \R \to \R_{+}$ be the function given by
    $$
    f(x,s) = e^{-s}\chi_{_{\{\|x\|_K\le s+n+1\}}}(x,s).
    $$ 
    Then
    
    \begin{equation}
        \begin{split}
            &\widehat{L}_{f} = \frac{(n+2)^{\frac{n}{2(n+1)}}\sqrt{n+1}}{e(n!)^{\frac{1}{n+1}}}L^{\frac{n}{n+1}}_{K}.
        \end{split}
    \end{equation}
\end{lemma}
\begin{proof}
    First, we have that
    \begin{equation*}
        \begin{split}
             \hspace{-0.85cm}\int_{\R^{n+1}} f(x,s) \,\dlat x\,\dlat s &= \int_{-(n+1)}^{+\infty} e^{-s} (s + n + 1)^n \vol(K) \, \dlat s\\
             &= n!\,e^{n+1}\vol(K),
        \end{split}
    \end{equation*}
    \begin{equation*}
        \begin{split}
            \hspace{0.45cm}\int_{\R^{n+1}} s f(x,s) \,\dlat x\,\dlat s &= \vol(K) \int_{-(n+1)}^{+\infty} s e^{-s} (s + n + 1)^n  \, \dlat s\\
            &= e^{n+1} \vol(K)\int_0^{+\infty} (t - n - 1) t^n e^{-t} \, \dlat t  = 0
        \end{split}
    \end{equation*}
and
    \begin{equation*}
        \begin{split}
            \hspace{0.3cm}\int_{\R^{n+1}} s^2 f(x,s) \,\dlat x\,\dlat s &= \vol(K) \int_{-(n+1)}^{+\infty} s^2 e^{-s} (s + n + 1)^{n} \;\dlat s\\
            &= \vol(K) e^{n+1} \int_0^{+\infty} (t - (n+1))^2 t^n e^{-t} \; \dlat t\\
            &= \vol(K) e^{n+1} (n+1)! 
        \end{split}
    \end{equation*}
Moreover, for every $i,j \in \{1,...,n\}$,
    \begin{equation*}
        \begin{split}
           \hspace{0.85cm}\int_{\R^{n+1}} x_i x_j f(x,s) \,\dlat x\,\dlat s &= \int_{-(n+1)}^{+\infty}  e^{-s} (s + n + 1)^{n+2}  \, \dlat s \int_K x_i x_j \,\dlat x\\
           &= (n+2)!\,e^{n+1} \int_K x_i x_j \,\dlat x,       
        \end{split}
    \end{equation*}
    \begin{equation*}
        \begin{split}
         \hspace{0.8cm}\int_{\R^{n+1}} x_i s f(x,s) \,\dlat x\,\dlat s  = \int_{-(n+1)}^{+\infty} s e^{-s} (s + n + 1)^n  \, \dlat s \int_K x_i  \,\dlat x = 0.
        \end{split}
    \end{equation*}
and
    \begin{equation*}
        \begin{split}
            \int_{\R^{n+1}} x_i f(x,s) \,\dlat x\,\dlat s &= \int_{-(n+1)}^{+\infty} e^{-s} (s + n + 1)^{n+1} \, \dlat s \int_K x_i \,\dlat x\\
            &= (n+1)! \, e^{n+1} \int_K x_i \,\dlat x=0 
        \end{split}
    \end{equation*}
Hence, 
\begin{equation*}
    \cov(f) = \begin{pmatrix}
    \begin{array}{c|c}
    (n+2)(n+1) \int_{K} \frac{x_i\, x_j}{\vol(K)}\,\dlat x& 0 \\
 \hline
   0 & (n+1) \\
 \end{array}
\end{pmatrix}.
\end{equation*}
Finally, it is not difficult to check that 
$$
h(f) = -\frac{\int f \log f}{\int f} = \frac{\int_{\R^{n+1}} s f(x,s)\,\dlat x \,\dlat s}{\int f}=0.
$$
Putting all together, we get
$$
\widehat{L}_{f} = \left(\frac{e^{-h(f)}}{\int_{\R^n}f(x) \,\dlat x}\right)^\frac{1}{n+1}\bigl(\det\cov(f)\bigr)^\frac{1}{2(n+1)}  = \frac{(n+2)^{\frac{n}{2(n+1)}}\sqrt{n+1}}{e(n!)^{\frac{1}{n+1}}}L^{\frac{n}{n+1}}_{K}.
$$

\end{proof}

\begin{proof}[Proofs of Propositions \ref{p:geometrical_implies_functional} and \ref{p:functional_implies_geometrical}] We start by proving i) in both propositions. Let $f : \R^n\to \R_{+}$ be an integrable log-concave function. Writing $f(x) = e^{-\varphi(x)}$, where $\varphi : \R^n \to \R$ is a convex function, we may see $f$ as a limit of $(1/m)$-concave functions, with $m >0$, since
    $$
    f(x) = \lim_{m \to +\infty} \left(1 - \varphi(x)/m \right)_+^m,
    $$
    and $f_{m}(x) = \bigl(1 - \varphi(x)/m\bigr)_+^m$ is a $(1/m)$-concave function. Moreover, if we assume that $f$ is even, then $f_m$ is even too. 
    
    Hence, setting $K_{m}(C,f_m^{1/m})$ as before, we shall consider two cases. If $f$ is even, taking $C = \B^m_{\infty}$, from Conjecture \ref{GC:symmetric} together with Lemma \ref{l:iso_constant_of_revbodies} we get that 
         $$
     L_{\B^{m+n}_{\infty}}^{2(m+n)} \geq L_{K_{m}(\B^m_{\infty},f_m^{1/m})}^{2(m+n)} =  \frac{\left(\int_{\R^n} f_m(x)^{(m+2)/m} \,\dlat x\right)^m}{\left(\int_{\R^n} f_m(x)\,\dlat x\right)^{m+2}} \det\cov(f_m) \, L_{\B^{m}_{\infty}}^{2m}.
     $$
     Therefore, for every $m \geq 0$ we get that 
     $$
      \frac{\left(\int_{\R^n} f_m(x)^{(m+2)/m} \,\dlat x\right)^m}{\left(\int_{\R^n} f_m(x)\,\dlat x\right)^{m+2}} \det\cov(f_m) \leq \frac{L_{\B^{m+n}_{\infty}}^{2(m+n)}}{L_{\B^{m}_{\infty}}^{2m}} = 12^{-n}.
     $$
    Now, we deduce by taking limits that 
    $$
     \widehat{L}_{f}^{2n} = \lim_{m \to +\infty}  \frac{\left(\int_{\R^n} f_m(x)^{(m+2)/m} \,\dlat x\right)^m}{\left(\int_{\R^n} f_m(x)\,\dlat x\right)^{m+2}} \det\cov(f_m) \leq 12^{-n}.
    $$
    from where the even case of Conjecture \ref{c:Functional_Slicing_even} iii) follows.
    \\
    
    For the general case, setting $C = \Delta^m$ we obtain that
    $$
    L_{\Delta^{m+n}}^{2(m+n)} \geq L_{K_{m}(\Delta^m,f_m^{1/m})}^{2(m+n)} =  \frac{\left(\int_{\R^n} f_m(x)^{(m+2)/m} \,\dlat x\right)^m}{\left(\int_{\R^n} f_m(x)\,\dlat x\right)^{m+2}} \det\cov(f_m) \, L_{\Delta^{m}}^{2m}.
    $$
    Thus, for every $m \geq 0$, we have that
    $$
    \frac{\left(\int_{\R^n} f_m(x)^{(m+2)/m} \,\dlat x\right)^m}{\left(\int_{\R^n} f_m(x)\,\dlat x\right)^{m+2}} \det\cov(f_m) \leq \frac{L_{\Delta^{m+n}}^{2(m+n)}}{L_{\Delta^{m}}^{2m}}.
    $$
    Finally, we obtain by taking limits that
    $$
    \widehat{L}_{f}^{2n} = \lim_{m \to +\infty}  \frac{\left(\int_{\R^n} f_m(x)^{(m+2)/m} \,\dlat x\right)^m}{\left(\int_{\R^n} f_m(x)\,\dlat x\right)^{m+2}} \det\cov(f_m) \leq \lim_{m \to +\infty} \frac{L_{\Delta^{m+n}}^{2(m+n)}}{L_{\Delta^{m}}^{2m}} = e^{-2n},
    $$
    which is the statement of Conjecture \ref{c:Functional_Slicing} iii).
    \\
    
Proposition \ref{p:geometrical_implies_functional} ii) can be immediately proved by taking the characteristic function of a centrally symmetric convex body. 

We end by proving ii) in Proposition \ref{p:functional_implies_geometrical}. To this regard, let $K \subset \R^n$ be a convex body and let $f:\R^n \times \R \to \R_{+}$ be the log-concave function given by 
    $$
     f(x,s) =e^{-s}\chi_{_{\{\|x\|_K\le s+n+1\}}}(x,s).
    $$
    Using the computations previously done in Lemma \ref{l:relationship_conjectures}, we get that
    $$
    \widehat{L}_{f} = \frac{(n+2)^{\frac{n}{2(n+1)}}\sqrt{n+1}}{e(n!)^{\frac{1}{n+1}}}L^{\frac{n}{n+1}}_{K}.
    $$
    Thus, since we are assuming that Conjecture \ref{c:Functional_Slicing} iii) holds true for every integrable log-concave function in $\R^{n+1}$, we have $\widehat{L}_{f}\le 1/e$ thus
    $$
    L_K \leq \frac{(n!)^{\frac{1}{n}}}{\sqrt{n + 2}\,(n + 1)^{\frac{n+1}{2n}}} = L_{\Delta^n}.
    $$
    \end{proof}

{\bf Acknowledgments:} This work was started while the second named author was visiting the LAMA at Université Gustave Eiffel. He would like to express his gratitude to all its researchers and staff members for their hospitality.

\end{document}